\documentclass[12pt,a4paper]{article}
\usepackage{a4,amssymb,amsmath,amsthm,url,color,lineno}
\usepackage{bezier,amsfonts,amssymb,graphicx,amsthm,url}
\usepackage[english]{babel}
\title{Notes on Spreads of Degrees in Graphs}
\date{}
\begin{document}
\newtheorem{theorem}{Theorem}[section]
\newtheorem{definition}{Definition}[section]
\newtheorem{proposition}[theorem]{Proposition}
\newtheorem{corollary}[theorem]{Corollary}
\newtheorem{lemma}[theorem]{Lemma}
\DeclareGraphicsExtensions{.pdf,.png,.jpg}

\author{Yair Caro \\ Department of Mathematics\\ University of Haifa-Oranim \\ Israel \and Josef  Lauri\\ Department of Mathematics \\ University of Malta
\\ Malta \and Christina Zarb \\Department of Mathematics \\University of Malta \\Malta }

\maketitle
\begin{abstract}
Perhaps the very first elementary exercise one encounters in graph theory is the result that any graph on at least two vertices must have at least two vertices with the same degree. There are various ways in which this result can be non-trivially generalised. For example, one can interpret this result as saying that in any graph $G$ on at least two vertices there is a set $B$ of at least two vertices such that the difference between the largest and the smallest degrees (in $G$) of the vertices of $B$ is zero. In this vein we make the following definition. For any $B\subset V(G)$, let the spread $sp(B)$ of $B$ be defined to be the difference between the largest and the smallest of the degrees of the vertices in $B$. For any $k\geq 0$, let $sp(G,k)$ be the largest cardinality of a set of vertices $B$ such that $sp(B)\leq k$. Therefore the first elementary result in graph theory says that, for any graph $G$ on at least two vertices, $sp(G,0)\geq 2$.

In this paper we first give a proof of a result of Erd\" os, Chen, Rousseau and Schelp which generalises the above to $sp(G,k)\geq k+2$ for any graph on at least $k+2$ vertices. Our proof is short and elementary and does not use the famous Erd\" os-Gallai Theorem on vertex degrees. We then develop lower bounds for $sp(G,k)$ in terms of the order of $G$ and its minimum, maximum and average degree. We then use these results to give lower bounds on $sp(G,k)$ for trees and maximal outerplanar graphs, most of which we show to be sharp.    
\end{abstract}

\section {Introduction}

One of the most fascinating aspects of combinatorics is that a trivial statement can be turned into a non-trivial result or even a very difficult problem by some very natural generalisation. Very often this involves the use of the pigeonhole principle. An application of this principle gives what we call the first elementary result in graph theory: any graph on at least two vertices has at least two vertices with the same degree. This result has been generalised in various directions, for example: a characterisation of those graphs which have only one repeated pair of degrees  \cite{behzad1967no}, and a characterisation of graphic sequences, that is, those sequences of positive integers which can be realised as the degree sequence of some graph  \cite{hakami1962realizability,havelremark} .
	
In this paper we consider the following generalisation of the first elementary result in graph theory, introduced in \cite{CERS}. Let $G=(V,E)$ be a graph.  For $B$ a subset of the vertex set $V$, we define the \emph{spread} of $B$ as $sp(B) = \{ \max(deg(u))-\min(deg(v)):  u,v \in B\}$, where the degrees are the degrees in graph $G$.
We then let, for an integer $k \geq 0$, $sp(G,k)$  be   $\max \{ |B| : sp(B) \leq k \}$, namely the largest cardinality of a subset of vertices of $G$ with spread at most $k$.  

The first elementary result of graph theory therefore says that, if $G$ has order at least 2, then $sp(G,0)\geq 2$. The number $sp(G,k)$ is also a generalisation of the maximum occurrence of a value in the degree sequence of a graph, as defined in \cite{caro2009repetition} and denoted by $rep(G)$, since $rep(G)=sp(G,0)$.  
  
The result $sp(G,0)\geq 2$ was extended to general spreads in \cite{CERS} where the following theorem was proved. 

\begin{theorem} [Erd\"os, Chen, Rousseau and Schelp] \label{cers}
Let $G$ be a graph on $n \geq k+2$ vertices, then $sp(G,k) \geq k+2$.
\end{theorem}

In this paper, in Section \ref{sec:generalbounds},  we give a short and elementary proof  of Theorem \ref{cers} avoiding the use of the Erd\"os-Gallai theorem. 
Then, in the same section, we develop a lower bound for $Sp(G,k)$  in terms of the parameters $n$, $\delta$, $d$, $\Delta$,  which are respectively the number of vertices, the minimum degree, the average degree and the maximum degree of the graph $G$.  Doing so we generalize a basic lemma and technique introduced in \cite{caro2009repetition}.

Then in Section \ref{families}, we consider the sharpness of the lower bounds obtained in Section \ref{sec:generalbounds} attained by trees and maximal outer-planar graphs (abbreviated to MOPs). We conclude in Section \ref{sec:conclusion} with some concluding remarks and open problems.

\section{Bounds for $sp(G,k)$} \label{sec:generalbounds}

The proof given in \cite{CERS} of Theorem \ref{cers} 
uses the celebrated Erd\" os-Gallai characterization of graphic sequences \cite{Erdos1960graphs}.  Here we give a very short and elementary (avoiding Erd\" os-Gallai theorem) alternative proof. 

\medskip\noindent
\emph{Proof of Theorem \ref{cers}.}

Suppose, on the contrary,that $G$ is a graph with $n = m( k+1)+ r$ vertices, $m \geq 1$,  $1 \leq r \leq k+1$ with $sp(G,k) \leq k+1$.
Let the vertices of $G$ be $\{v_1,\ldots,v_{m(k+1)}, v_{m(k+1) +1} = u_1,\ldots,,v_{m(k+1) +r} = u_r \}$.  By assumption on $G$, $deg(v_{ j +k +1}) \geq deg(v_j ) +k +1$ for $j = 1,\ldots, n - k - 1$, as each interval has $k+2$ vertices and we assumed that $spG) \leq k+1$.  Hence in particular, \[deg( u_j) = deg( v_{m(k+1) +j} ) \geq m(k+1) +deg(v_j) \mbox{                  (1)}\] for $j = 1,..,r$.

How many vertices among $v_1,\ldots,v_r$ can $u_1,\ldots,u_r$ be adjacent to?

Clearly each $u_ j$ can be adjacent among $v_{r +1},\ldots,v_n$ to at most $n - r -1$ vertices and \[n-r-1=m(k+1)+r-r-1=m(k+1)-1 \mbox{       (2)}.\]  Hence, using (1) and (2), $u_j$  is adjacent to at least $d(u_j) - m(k+1) +1 \geq m(k+1)) +deg(v_j) - m(k+1) +1 = d( v_j) +1$ vertices among $v_1,\ldots,v_r$.  But then consider the bipartite graph $H$ with $v_1,\ldots,v_r$ on one side and $u_1,\ldots,u_r$ on the other side.  Clearly, if $deg_H(u_j)$ denotes the degree of vertex $u_j$ in $H$, we obtain \[\sum_{j=1}^{j=r}deg(v_j)  \geq e( H) \geq \sum_{j=1}^{j=r}deg_H(u_j) \geq\sum_{j=1}^{j=r}deg(v_j) +1,\] a contradiction .
\hfill $\square$

Before stating our main results in this section , we observe that $sp(G,k)=sp(\overline{G},k)$ since for any two vertices $u, v \in V(G)$,\[deg_{\overline{G}}(u)-deg_{\overline{G}}(v)=(n-1-deg_G(u))-(n-1-deg_G(v))=deg_G(v)-deg_G(u).\]
\begin{theorem} \label{lbsp}
Let $G$ be a graph on n vertices average degree $d$, minimum degree $\delta$ and maximum degree $\Delta$.  
Then:
\begin{enumerate}
\item{$sp(G,k) \geq \max\{\frac{n(k+1)}{2d - 2\delta +k+1},\frac{n(k+1)}{2\Delta - 2d +k+1}\}$.}
\item{$sp(G,k) \leq (k+1)sp(G,0)$.}
\end{enumerate}
\end{theorem}

\begin{proof}

Let $r = sp(G,k)$ and set $n = rt +b$, where $0 \leq b \leq r-1$, and consider the intervals \[I_1 =
[\delta,\delta+k], I_2 = [\delta+k+1 , \delta+2k+1],\ldots,\]\[ I_t = [\delta+(t-1)(k+1) , \delta+t(k+1) - 1], I_b = [ \delta +t(k+1),\ldots,n-1].\]

Each interval $I_ j$ contains at most $r$ vertices from $V(G)$ for otherwise $sp(G,k) \geq r+1$.  There are $t$ such intervals containing at most $rt$ vertices altogether and at least $b$ elements from the interval $I_b$ so that the total number of vertices is $rt +b = n$.

The smallest degree sum is achieved when we take exactly $r$ elements in each interval $I_ j$
with value $\delta + (j-1)(k+1)$ and the $b$ elements in $I_b$ equals $\delta +t(k+1)$, so that the total sum of degrees is 
\begin{align*}
2e(G) &= dn\\
& \geq r [\delta + (\delta+k+1) +\ldots+ (\delta +(t-1)(k+1)) + b(\delta+t(k+1))]\\
&= \frac{2rt\delta
+rt(t-1)(k+1)}{2} + b(\delta+ t(k+1))\\
&=  \frac{rt[2\delta + (t-1)(k+1)]}{2} + b(\delta+ t(k+1))\\
& =  \frac{(n-b)[2\delta + (t-1)(k+1)]}{2} + b(\delta+ t(k+1))\\
& =n\delta + \frac{n(t-1)(k+1)}{2} -b\delta - \frac{b(t-1)(k+1)}{2} +b\delta + bt(k+1)\\
& =n\delta + \frac{n(n-b)}{r -1)(k+1)} - \frac{bt(k+1)}{2} +\frac{b(k+1)}{2} + bt(k+1)\\
& = n\delta +  \frac{n(k+1)(\frac{n}{r} -1)}{2} - \frac{nb(k+1)}{2r} +
\frac{bt(k+1)}{2} + \frac{b(k+1)}{2} \\
&= n\delta + - \frac{nb(k+1)}{2r} +
\frac{rbt(k+1)}{2r} + \frac{rb(k+1)}{2r} \\
&= n\delta +  \frac{n(k+1)(\frac{n}{r} -1)}{2}- \frac{nb(k+1)}{2r}
+ \frac{(n-b)b(k+1)}{2r} + \frac{rb(k+1)}{2r}\\
& =n\delta +   \frac{n(k+1)(\frac{n}{r} -1)}{2}- \frac{b^2(k+1)}{2r} + \frac{rb(k+1)}{2r}\\
& = n\delta +   \frac{n(k+1)(\frac{n}{r} -1)}{2} + \frac{b(r-b)(k+1)}{2r} \mbox{   (3)}\\
&\geq n\delta +
 \frac{n(k+1)(\frac{n}{r} -1)}{2}
\end{align*}
  taking $b=0$ in (3).  Hence $dn \geq n\delta + \frac{n(k+1)(\frac{n}{r} -1)}{2}$ which after rearranging gives $r \geq \frac{n(k+1)}{2d- 2\delta +k+1}$, the first expression.

Also since $sp(G,k) = sp(\overline{G},k)$ and using $\overline{d}=n-1-d$ and $\overline{\delta}=n-1-\Delta$, we get 
\[sp(G,k)=sp(\overline{G},k) \geq\frac{n(k+1)}{2\Delta - 2d +k+1}.\]

2.  For $sp(G,k)$, the spread is determined by a set of vertices with degrees $p, p+1, \ldots,p+k$ respectively.  Let $S_i$ be the set of vertices of degree $i$ in this set.  Then $0 \leq |S_i| \leq sp(G,0)$, hence $sp(G,k) \leq (k+1)sp(G,0)$.
\end{proof}

\textbf{Remark:}  Observe that in equation (3) we used $b = 0$,  but if we substitute $b =n-rt$,  then after some further algebra  we get

\[r \geq \frac{2n(\delta - d +t(k+1) ) }{ t(t+1)(k+1)}  \mbox{     (4)}.\]        

This will  prove useful once we have a lower bound $r^*$ on $r$ using Theorem \ref{lbsp}  and an upper bound $t^*$ on $t$ since from  $n = rt+b$ we get $n \geq r^*t +b$  hence  $\frac{n -b}{r^*}= t^*\geq t$.

Clearly $r$ is at least the minimum in equation (4) over all $t$ such that $1 \leq t \leq t^*$.

We shall use this remark several times in section \ref{families}.

\section{Realisation of the lower bounds in certain families of graphs.} \label{families}

The characterisation of graphic sequences in general given in \cite{Erdos1960graphs,hakami1962realizability,havelremark} is too wide to force restrictions on the degree sequence so that the bounds of the Theorem are attained. It is therefore interesting to investigate classes of graphs whose structure imposes such restrictions. In this section
we show that trees and maximal outerplanar graphs come very close to having this required structure: for both classes, their average degree d which appears in the bound of Theorem \ref{lbsp}, is known in terms of the number $n$ of vertices, and their structure forces severe restrictions on the possible degrees which their vertices can have.



\subsection{Trees}

\begin{theorem}
Let $k \geq 0$ and $T$ be a tree on $n \geq k+2$ vertices.  Then 
\begin{enumerate}
\item{$sp(T,0) = rep(T) \geq \lceil \frac{n}{3} \rceil$ which is sharp for $n =1\pmod 3$.} 
\item{For $k \geq 1$, $sp(T,k) \geq \frac{nk+2}{k+1}$ and this is sharp.}
\end{enumerate}
\end{theorem}

\begin{proof}

1.  The case $k = 0$  is from \cite{caro2009repetition} and sharpness for $n =1\pmod 3$ is achieved by a tree made up of a path on $3k+2$ vertices, with a path of two edges attached to the vertices $v_3 \ldots v_{3k}$ to give a tree $T$ on $3k+2+2(3k-2) = 9k-2$ vertices.  This gives $3k$ vertices of degree 1, $3k$ vertices of degree 2 and $3k-2$ vertices of degree 3, and hence $sp(T,0)=3k=\frac{n+2}{3} $.

1.  For a tree, $\delta = 1$ and   $d = \frac{2(n-1)}{n}$, and  substituting into  Theorem \ref{lbsp} with $k=1$ gives  \[\frac{2n}{ 4 - \frac{4}{n }- 2 +2} = \frac{2n^2}{4(n-1)} = \frac{n^2}{2(n-1)}> \frac{n+1}{2}.\] Hence in $n=rt+b$ we just have $t=1$ otherwise $rt>n$.  Furthermore since $sp(T,k+1) \geq sp(T,k)$ we get that for all $k \geq 1$ we may assume $t=1$.

For trees and $k \geq 1$  the lower bound (4) (with $t = 1$)  gives \[r \geq \frac{2n(\delta - d +t(k+1))}{t(t+1)(k+1)} = \frac{2n(1  2 +2/n +k +1)}{2(k+1)} = \frac{n(k +\frac{2}{n})}{k+1}= \frac{nk+2}{k+1}.\] 
This is sharp for every $k \geq 1$, as can be seen with trees having degrees only 1 and $k+2$ using the following following equations with $n_ j$ being the number of vertices of degree $j$:
\begin{enumerate}
\item{Vertex counting: $n_1 + n_{k+2} = n$}
\item{Edge counting: $ n_1 +(k+2)n_{k+2} = 2n -2$}
\end{enumerate}

Then solving for $n_1$ we get $n_1 = \frac{nk+2}{k+1}  =  sp(T ,k)$ as required.
\end{proof}


\subsection{Maximal Outerplanar Graphs}

We now consider  maximal outerplanar graphs.  In general, for a maximal outerplanar graph $G$ on $n$ vertices,  bound (1) gives \[sp(G,k) \geq \frac{(k+1)n}{\frac{4(2n-3)}{n}-4+k+1}  \geq  \frac{(k+1)n}{\frac{(5+k)n-12}{n}} \geq  \frac{(k+1)n^2}{(5+k)n-12} > \frac{(k+1)n}{5+k}.\] 

We define \[MOP(n,k) = \min \{ sp(G,k): \mbox{ where $G$ ranges over all maximal outer-planar graphs on $n$ vertices}\}.\]  We prove the following results.

\begin{theorem}

\begin{enumerate}
For maximal outerplanar graphs
\item{$MOP(n,0)  > \frac{n}{5}$.}
\item{$MOP(n,1) \geq \frac{n}{3}+1$.}
\item{$\frac{5n+19}{11} \geq MOP(n,2) \geq \frac{4n}{9}$.}
\item{For $k \geq 3$, $MOP(n,k) \geq \frac{(k-2)n}{k-1}$.}
\end{enumerate}

Bounds 1 and 2  are sharp up to small additive constants.
\end{theorem}

\begin{proof}

$\mbox{  }$\\

\noindent 1.  $MOP(n,0)   = \min \{ rep(G)  : \mbox{where $G$ is a maximal outer-planar graph on $n$ vertices} \}$.  Bound (1) gives $MOP(n,0) > \frac{n}{5}$ which is the same as the lower bound given for $rep(G)$ in \cite{caro2009repetition}.  The construction given in  \cite{caro2009repetition} gives $rep(G)=\frac{n-4}{5}+2=\frac{n}{5}+\frac{6}{5}$ when $n =4 \pmod{10}$.

\bigskip

\noindent 2.    For $k=1$, the above observation gives $MOP(n,1) >  \frac{n}{3}.$  Hence we may use $t = 1 ,2$  and for $t = 2$  and $k = 1$ we get using bound ( 4 )    \[MOP(n,1) \geq \frac{2n(2 - \frac{4n -6}{n} +4 )}{12} = \frac{n(2n +6)}{6n} = \frac{2n+6}{6} = \frac{n}{3}+1.\]

The following construction  realises this bound up to a constant.

Arrange three sets of vertices $U = \{ u_1,\ldots,u_{p-1}\}$, $V = \{ v_1,\ldots,v_p\}$ and $W = \{
w_1,\ldots,w_{p-1}\}$.  $U$ will be the upper vertices, $V$ will be in the middle vertices and $W$ the bottom ones.

Let $u_i$ be connected to $v_i$ and $v_{i+1}$; let  $w_i$ be connected to $v_i$ and $v_{i+1}$ and to $w_{i-1}$ and $w_{i+1}$, except $w_1$ which is  only connected to $w_2$, and $w_{p-1}$ which is only connected to $w_{p-2}$.  Let $v_i$ be also connected  to $v_{i-1}$ and $v_{i+1}$ (except the first and the last).  Figure \ref{MOP1fig} shows an example of  this construction.

\begin{figure}[h] 
\centering
\includegraphics{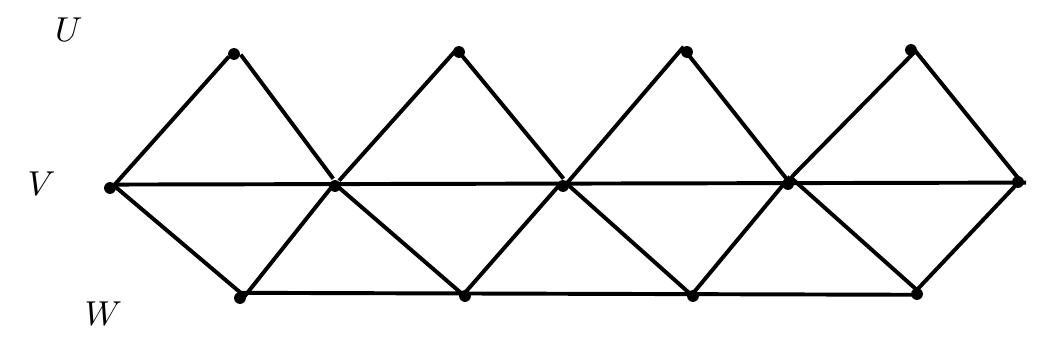} 
\caption{The above construction for $p=5$} \label{MOP1fig}

\end{figure}

This is a maximal outerplanar graph with $p-1$ vertices of degree 2, four  vertices of degree 3, $p-3$
vertices of degree 4 and $p-2$ vertices of degree 6.  So we have $3p-2$ vertices and $sp(G,1) = p+3 = \frac{(3p-2) + 11}{3}=\frac{n+11}{3}$,  which differs from the lower bound by $\frac{8}{3}$.

\bigskip

\noindent 3.  For $MOP(n,2)$, bound (1) gives $MOP(n,2) \geq \frac{3n}{7}$ while bound (4) (with t=2) gives $MOP(n,2) \geq \frac{4n}{9}$.  We shall present a construction showing the bound $\frac{5n+19}{11}$ later on.

\bigskip

\noindent 4.    In \cite{jao2012vertex}, the authors define $\beta_k(n)$ to be the maximum number of vertices of degree at least $k$ amongst all maximal planar graph of order $n$.   They show that for $k \geq 6$ and $n \geq k+2$, \[\beta_k(n) \geq \left\lfloor \frac{n-6}{k-4} \right\rfloor.\]

Since $\delta=2$ for any maximal outerplanar graph, it follows that \[MOP(n,k) \geq n - \beta_{k+3}(n) \geq n -  \left\lfloor \frac{n-6}{k-1} \right\rfloor \geq \frac{n(k-2)+6}{k-1} \geq \frac{(k-2)n}{k-1}.\]

 Let $B$ the graph in Figure \ref{FigB}.  

\begin{figure}[h] 
\centering
\includegraphics{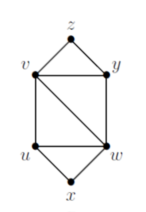} 
\caption{The graph B} \label{FigB}

\end{figure}

The graph $B'$ is obtained by replacing the edges $ux$ and $yz$ by  paths $P$ and $Q$ respectively, containing $k-3$ internal vertices each.  The vertex $v$ is joined to every vertex on $P$ and the vertex $w$ is joined to every vertex in $Q$. We then create the graph $F^t_k$ by taking the union of  $t$ copies of the graph $B'$.  Figure \ref{FigB2} shows and example with $k=4$ and $t=3$.  The graph $F^t_k$ has $n= 2t(k-1)+6$ vertices.  Such graphs have $2t+2$ vertices of degrees 2 and  $2t(k-3)+2$ vertices of degree 3, two vertices of degree 4 and $2t$ vertices of degree $k+3$.  This gives \[MOP(n,k)=2t(k-2)+6 = 2(k-2)\left(\frac{n-6}{2(k-1)} \right)+6= \frac{(k-2)n+6}{k-1}.\] 

\begin{figure}[h!] 
\centering
\includegraphics{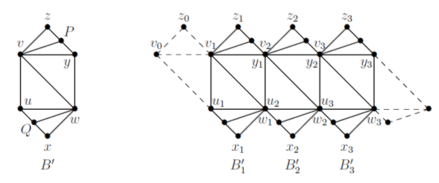} 
\caption{The graph $F^3_4$} \label{FigB2}

\end{figure}

\end{proof}

\newpage
The following construction shows $\frac{5n+19}{11} \geq MOP(n,2) $ for n$ = 5 \pmod{11}$,  and hence  for other values of $n \pmod{11}$ can be completed by adding at most 10 vertices.  Hence $MOP(n,2)  \geq  \frac{5n}{11} + c(n,11)$  where $c(n,11)$  is a constant which depends on $n \pmod{11}$.

Consider a path $V=v_1,v_2,\ldots, v_{5p+3}$, a path $u_1,\ldots, u_p$ above it and the vertices $w_1,\ldots,w_{5p+2}$ below the path.  Let $u_1$ be adjacent to $v_1$ to $v_7$, and $u_p$ adjacent to $v_{5p-3}$ to $v_{5p+3}$, while for $1 \leq i \leq p-1$, $u_i$ is adjacent to $v_{5i-3}$ to $v_{5i+2}$. Vertex $w_j$ for $1 \leq j \leq 5p+2$ is adjacent to $v_j$ and $v_{j+1}$.  This gives a total of $n=11p+5$ vertices: $p$ vertices of degree 8, $p-1$ vertices of degree 6, $4p+2$ vertices of degree 5, 2 vertices of degree 3 and $5p+2$ vertices of degree 2.  Thus $Sp(n,2)=5p+4=\frac{5(n-5)}{11} +4=\frac{5n+19}{11}$.  Figure \ref{MOP2fig} shows an example of this constuction.

\begin{figure}[h!] 
\centering
\includegraphics{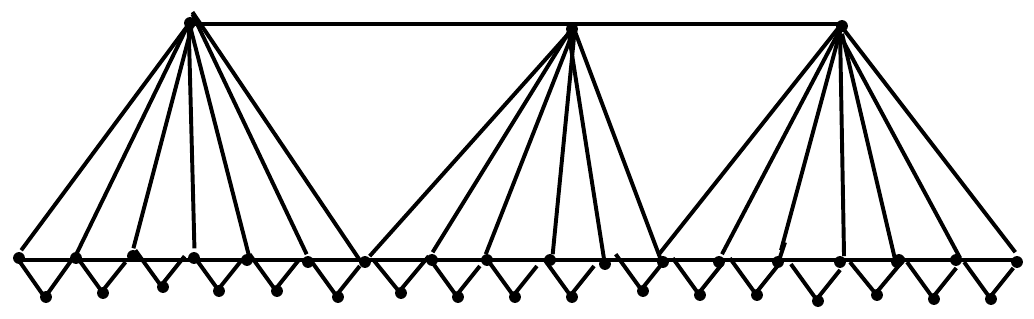} 
\caption{The above construction for $p=3$} \label{MOP2fig}

\end{figure}

\noindent
\textbf{\textit{Some final remarks.}} It might be useful to try and see at this point where applying Theorem \ref{lbsp} does not work even for trees and maximal outerplanar graphs. Let us elaborate on the simple observation we made in the introduction to this section. For trees, the phenomenon we described occurs  for $k=0$ because the proof of the Theorem would require degrees 1, 2, 3 with equal classes, but already for $k=2$ with degrees 1 and 4 in equal classes the average degree would be 3 which is impossible for trees.  Hence for $k\geq 2$ it all works out, with sharpness coming from $t=1$ in (4) giving trees of degrees 1 and $k+2$. 
	
And again, for maximal outerplanar graphs, for $k=2$ we should have degrees 2, 5, 8 with equal classes which will only give $n/3$ and $d=5$, but this is too large as $d=4$ for maximal outerplanars. So letting $b=0$ in the proof of Theorem \ref{lbsp}, which anyway would give $3n/7$, would force two big equal classes and the remainder. Using (4) of Theorem \ref{lbsp} with $t=2$ gives $4n/9$ which would be possible if we could find maximal outerplanars with 4n/9 vertices of degree 2 and degrees 5 and $n/9$ vertices of degree 8, but we could not find such constructions yet.

\section{Conclusion} \label{sec:conclusion}

The results presented in this paper naturally lead to an unanswered question and to the most likely next class for which one can investigate whether the spread  attains the bounds of Theorem \ref{lbsp}.

The obvious unanswered question is determining the best lower bound for $MOP(n,2)$, that is, the minimum spread $sp(G,2)$ among all maximal outerplanar graphs on $n$ vertices. 
We know, by the general lower bound given by bound (4), that $sp(G,2)$ is at least $4n/9$. 
While for the other spreads we considered in Section \ref{families} we could get close to the  bound given by (4) up to small additive constants, for $MOP(n,2)$ the family of outerplanar graphs $G$ on $n$ vertices with lowest value for $sp(G,2)$ which we could find gave $sp(G,2)$ approaching $5n/11$.  
\medskip 

\noindent
\emph{Problem 1}: Determine the correct order of magnitude of $MOP(n,2)$.

\medskip

One can  also consider maximal planar graphs.  We define \[MP(n,\delta,k)  =  \min \{ sp(G,k) : G \mbox{ \footnotesize{is maximal planar graph on $n$ vertices and minimum degree $\delta$}}\}.\]

\noindent
\emph{Problem 2}:  Determine $MP(n,\delta,k)$ for $\delta = 3 ,4 ,5$ and  $k \geq 0$.



\bibliographystyle{plain}

\bibliography{spreadfinal}

\end{document}